\documentclass[11pt]{amsart}

\usepackage{amssymb,latexsym}
\usepackage[usenames]{color}

\usepackage{mathrsfs}
\usepackage{amsfonts}
\usepackage{amsmath}
\usepackage{graphicx}
\usepackage{amsthm,amscd}
\usepackage{calc}
\usepackage{hyperref}
\usepackage{setspace}
\usepackage{color}
\usepackage[dvipsnames]{xcolor}

\newtheorem{theorem}{Theorem}[section]

\newtheorem{corollary}[theorem]{Corollary}
\newtheorem{definition}[theorem]{Definition}

\newtheorem{lemma}[theorem]{Lemma}

\newtheorem{proposition}[theorem]{Proposition}
\newtheorem{remark}[theorem]{Remark}

\renewcommand{\bar}{\overline}

\newcommand{\pa}{\partial}
\renewcommand{\phi}{\varphi}

\newcommand{\ka}{K\"ahler }

\newcommand{\ke}{K\"ahler-Einstein }

\newcommand{\ga}{\alpha}
\newcommand{\gb}{\beta}

\newcommand{\gl}{\lambda}
\newcommand{\gm}{\mu}

\newcommand{\bfC}{{\mathbf C}}

\newcommand{\bfR}{{\mathbf R}}

\newcommand{\bari}{{\overline i}}
\newcommand{\barj}{{\overline j}}
\newcommand{\bark}{{\overline k}}

\newcommand{\barp}{{\overline p}}

\newcommand{\barz}{{\overline z}}

\newcommand{\barpartial}{{\overline \partial}}

\newcommand{\tu}{{\widetilde u}}

\newcommand{\mapright}[1]{\smash{\mathop{   \hbox to 0.7cm{\rightarrowfill}}
  \limits^{#1}}}

\newcommand{\Ric}{\operatorname{Ric}}

\newcommand{\grad}{\mathrm{grad}}
\newcommand{\Fut}{\mathrm{Fut}}

\title{Coupled Sasaki-Ricci solitons}
\author{Akito Futaki and Yingying Zhang}
\address{Yau Mathematical Sciences Center, Tsinghua University, Haidian district, Beijing 100084, China}
\email{futaki@tsinghua.edu.cn}
\address{Yau Mathematical Sciences Center, Tsinghua University, Haidian district, Beijing 100084, China}
\email{zhangyingying@math.tsinghua.edu.cn}
\date{December 15, 2018}

\begin{document}

\begin{abstract}Motivated by the study of coupled K\"ahler-Einstein metrics by Hultgren and Witt Nystr\"om \cite{HultgrenWittNystrom18} and coupled K\"ahler-Ricci solitons by Hultgren \cite{Hultgren17}, we study in this paper coupled Sasaki-Einstein metrics and coupled Sasaki-Ricci solitons. We first show 
an isomorphism between the Lie algebra of all transverse holomorphic vector fields and certain space of coupled
basic functions related
to coupled twisted Laplacians for basic functions, and obtain extensions of the well-known obstructions to the existence of K\"ahler-Einstein metrics to this coupled case. These results are reduced to coupled K\"ahler-Einstein metrics when the Sasaki structure is regular. Secondly we show the existence of toric coupled Sasaki-Einstein metrics when the basic Chern class is positive extending the work of Hultgren \cite{Hultgren17}.
\end{abstract}

\maketitle

\section{Introduction}
Motivated by the proposed study of coupled K\"ahler-Einstein metrics by Hultgren and Witt Nystr\"om \cite{HultgrenWittNystrom18} and coupled K\"ahler-Ricci solitons by Hultgren \cite{Hultgren17} we study in this paper coupled Sasaki-Einstein metrics and coupled Sasaki-Ricci solitons. Our work started in trying to understand their works from the viewpoint of the former studies \cite{futaki83.1}, \cite{futakimorita85}, \cite{futaki86}, \cite{futaki87}, \cite{futaki88} of K\"ahler-Einstein metrics, and hopefully would serve as a supplement to their papers since our results reduce to the case of coupled K\"ahler-Einstein metrics when the Sasaki structure is regular. 

Let $M$ be a Fano manifold and choose the K\"ahler class to be the anti-canonical class $K_M^{-1}$ or equivalently the first Chern class $c_1(M)$. A decomposition of $c_1(M)$ is a sum
\begin{equation}\label{Kdecomp}
c_1(M) = (\gamma_1 + \cdots + \gamma_k)/2\pi
\end{equation}
with K\"ahler classes $\gamma_1,\ \cdots,\ \gamma_k$. If we choose a K\"ahler form $\omega_\alpha$ representing $\gamma_\alpha$ for each $\alpha$, since both the Ricci form $\Ric(\omega_\alpha)$ and $\sum\limits_{\gb=1}^k \omega_\gb$ represent $2\pi c_1(M)$, there exists a smooth function $f_\alpha$ such that 
\begin{equation}\label{Kpotential}
\Ric(\omega_\alpha) - \sqrt{-1}\partial\barpartial f_\alpha = \sum_{\gb=1}^k \omega_\gb, \quad \alpha = 1,\ \cdots,\ k.
\end{equation}
Coupled K\"ahler-Ricci solitons are defined in \cite{Hultgren17} to be K\"ahler metrics with K\"ahler forms $\omega_\alpha$ representing $\gamma_\alpha$ such that, for each $\alpha$, $f_\alpha$ is a Hamiltonian Killing potential with respect to $\omega_\alpha$ so that $J\grad_\alpha f_\alpha$ is a Hamiltonian Killing vector field where $\grad_\alpha$ denotes the gradient with respect to $\omega_\alpha$. Coupled K\"ahler-Einstein metrics defined in \cite{HultgrenWittNystrom18} are the case when $f_\alpha$'s are all constant so that
\begin{equation}\label{coupled}
\Ric(\omega_1) = \cdots = \Ric(\omega_k) = \sum_{\gb=1}^k \omega_\gb.
\end{equation}

Now we consider Sasakian analogues of the above. Let $S$ be a compact Sasaki manifold with positive basic first Chern class $c_1^B(S)$,
which means $c_1^B(S)$ is represented by a positive basic $(1,1)$-form. We assume that the real dimension of S is $2m+1$. We define similarly a decomposition of $c_1^B(S)$ to be a sum 
\begin{equation}\label{Sdecomp}
 c_1^B(S) = (\gamma_1 + \cdots + \gamma_k)/2\pi
\end{equation}
of positive basic $(1,1)$ classes $\gamma_\alpha$. If we choose basic K\"ahler forms 
$\omega_\alpha$ representing $\gamma_\alpha$, there exist smooth basic functions $f_\alpha$ such that 
\begin{equation}\label{potential}
\Ric^T(\omega_\alpha) - \sqrt{-1}\partial_B\barpartial_B f_\alpha = \sum_{\gb=1}^k \omega_\gb, \quad \alpha = 1,\ \cdots,\ k
\end{equation}
where $\Ric^T$ denotes the transverse Ricci form, see \eqref{TRic} below.
We say $\omega_\alpha$'s are coupled Sasaki-Ricci solitons if for each $\alpha$, $f_\alpha$ is a Hamiltonian Killing potential for $\omega_\alpha$ and 
coupled Sasaki-Einstein metrics if $f_\alpha$ is constant
so that coupled Sasaki-Einstein metrics satisfy 
\begin{equation}\label{coupledS}
\Ric^T(\omega_1) = \cdots = \Ric^T(\omega_k) = \sum_{\gb=1}^k \omega_\gb.
\end{equation}
\begin{remark}
When $k=1$, the definition does {\it not} coincide with the usual definition of transverse \ke metric of a Sasaki-Einstein metric which is known to be equivalent to saying the transverse K\"ahler metric is Einstein. This is because the Riemannian metric of a Sasaki manifold naturally determines a transverse K\"ahler form written in the form $\frac12 d\eta$ for the contact $1$-form  with respect to the given Reeb vector field but the basic Chern class $c_1^B(S)$ 
need not be represented by $\frac12 d\eta$. However if we assume $c_1(D) = 0$ as a de Rham class for the contact distribution $D$, we have $c_1^B(S) = [\frac12 d\eta]$ and the definition coincides with the transverse K\"ahler-Einstein form of a Sasaki-Einstein metric. See Corollary ${\mathrm 7.5.26}$ in \cite{BGbook}.
\end{remark}

To study the coupled equations in Sasakian situation as described above we wish to extend results for Fano manifold as in \cite{futaki88}. 
Suppose we are given a compact Sasaki manifold $(S,g)$ with positive basic first Chern class
and a decomposition of $c_1^B(S)$ as \eqref{Sdecomp}. We also choose $\omega_\alpha$ in 
$\gamma_\alpha$ and then have $f_\alpha$ satisfying \eqref{potential}. We may normalize $f_\alpha$ so that 
\begin{equation}\label{normalize0}
e^{f_1}\omega_1^m= \cdots = e^{f_k}\omega_k^m.
\end{equation}
We define the twisted basic Laplacian $\Delta_{\alpha, f_\alpha}$ acting on basic smooth functions $u$ by
\begin{equation}\label{t.b.Lap}
\Delta_{\alpha, f_\alpha} u=\Delta_{\ga} u+(\grad_{\ga} u)f_\ga,
\end{equation}
where $\Delta_{\ga}=-\bar\pa^*_{B, \ga}\bar\pa_B$ is the basic Beltrami-Laplacian with respect to $\omega_\ga$
and $\grad_{\alpha}u = g_{\alpha}^{i\barj}\frac{\partial u}{\partial \barz^j} \frac{\partial}{\partial z^i}$, where $(x, z^1, \dots, z^m)$ is the local foliation chart, see section 2 below.

Recall that a compact Sasaki manifold $(S,g)$ is characterized by its Riemannian cone manifold $C(S)$ being a K\"ahler manifold. 
A transverse holomorphic vector field is a holomorphic vector field $X$ on the K\"ahler cone $C(S)$ which commutes with the extended Reeb vector field on $C(S)$, see section 2 for more details. It descends to a vector field on 
$S$, and also descends to a holomorphic vector field on local orbit spaces of the Reeb flow which we denote by the same letter
$X$.
Since $c_1^B(S)$ is positive the basic first cohomology is zero. Thus, for each basic K\"ahler form $\omega_\alpha$, there is a basic complex valued Hamiltonian function $u_{\alpha}$ such that
 $X = \grad_\alpha u_\alpha$.
Let $\mathfrak{h}^T(S)$ be the complex Lie algebra of all transverse holomorphic vector fields.

\begin{theorem}\label{thm2} Let $(S,g)$ be a compact Sasaki manifold with positive basic first Chern class $c_1^B(S)$
with decomposition satisfying \eqref{Sdecomp}. We choose basic K\"ahler forms and associated potential functions 
satisfying \eqref{potential} and \eqref{normalize0}. 
\begin{enumerate}
\item[(1)]
If non-constant complex valued basic functions $u_1, \dots, u_k$ satisfy
\begin{enumerate}
\item[(a)] $\grad_\ga u_\ga=\grad_\gb u_\gb$,\ \ $\ga, \gb=1, 2, \dots, k$.
\item[(b)] $-\Delta_{\ga, f_\ga} u_\ga=\lambda\sum\limits_{\gb=1}^k u_\gb$,\ \ for $\ga=1, 2, \dots, k$,
\end{enumerate}
then $\lambda\ge 1$. Moreover if $\lambda=1$, the complex vector field $V=\grad_\ga u_\ga=\grad_\gb u_\gb$ is a transverse holomorphic vector field.
\item[(2)]
The Lie algebra $\mathfrak{h}^T(S)$ of all transverse holomorphic vector fields is isomorphic to the set of all $k$-tuples of complex valued smooth functions $(u_1, \cdots u_k)$ satisfying (a) and (b) with $\lambda = 1$ endowed
with the Lie algebra structure with respect to Poisson bracket.
\end{enumerate}
\end{theorem}
\noindent
This theorem is a generalization of Theorem 2.4.3 in \cite{futaki88}, see also \cite{futaki87}. The case of $k=1$ of part (1) is 
an eigenvalue estimate of the twisted basic Laplacian. If $f_1,\ \cdots,\ f_k$ are all constant then $\Delta_{\ga, f_\ga} = \Delta_\ga$ is a real operator.
It follows that the gradient of both the real part and the imaginary part of $u_1,\ \cdots,\ u_k$ give a holomorphic vector field. Further, since if a
(real) Hamiltonian vector field is holomorphic it is necessarily Killing and the group of all isometries is compact, 
we obtain the following extension of a theorem of Matsushima \cite{matsushima57}.

\begin{corollary}\label{reductive}
If a compact Sasaki manifold $S$ admits coupled \ke metrics, then the Lie algebra $\mathfrak{h}^T(S)$ of transverse holomorphic vector fields is reductive. 
\end{corollary}
\noindent
This result was stated in \cite{HultgrenWittNystrom18} for coupled \ke metrics, but our proof would be more elementary.

As in other non-linear problems in K\"ahler geometry e.g. \cite{FO_reductive17}, \cite{Lahdili17}, \cite{FO_CahenGutt}, 
\cite{La Fuente-Gravy 2016_2}, 
a Lie algebra character obstruction as in \cite{futaki83.1} appears in pair with the reductiveness result of Theorem \ref{thm2}.
Using the isomorphism in Theorem \ref{thm2}, (2), 
we define
\begin{align}
\Fut: \mathfrak{h}^T(S)&\to \bfC\nonumber\\
V&\mapsto \sum\limits_{\ga=1}^k\frac{\int_S u_\ga\ \omega_\ga^m\wedge\eta}{\int_S \omega_\ga^m\wedge\eta}.\label{futaki}
\end{align}
We will see in section 4 that this definition is independent of the choice of $u_1,\ \cdots,\ u_k$  satisfying (a) and (b)
with $\lambda = 1$ above. 
\begin{theorem}\label{Futaki inv}
Suppose the basic \ka classes $\gamma_\ga$ give a basic decomposition $c_1^B(S)=\sum\limits_{\ga=1}^k\gamma_\ga/2\pi$.  
Then
$\Fut(V)$ is independent of the choice of basic \ka forms $\omega_\ga\in \gamma_\ga$, $\ga=1, \dots, k$.
Further if $S$ admits coupled Sasaki-Einstein metrics for the decomposition $c_1^B(S)=\sum\limits_{\ga=1}^k\gamma_\ga/2\pi$
then $\Fut$ identically vanishes.
\end{theorem}
The lifting of the Lie algebra character in \cite{futaki83.1} to a group character was obtained in \cite{futaki86}. 
This lifting is expressed in terms of Ricci forms, and if we replace them by K\"ahler forms it becomes the form of
the Monge-Ampere energy or Aubin's $J$-functional \cite{aubin84}. In \cite{Ding88} another form of lifting was obtained in the sense that it satisfies the cocycle conditions, and it is
now called Ding's functional. The definition of $\Fut$ in this paper uses the relationship of  \cite{futaki86} and \cite{Ding88}.
This seems to be already implicitly used in \cite{HultgrenWittNystrom18} and \cite{Hultgren17}. 

We also extend the existence results of Wang-Zhu \cite{Wang-Zhu04} and Hultgren \cite{Hultgren17} to toric coupled Sasaki-Einstein metrics as follows.
\begin{theorem}\label{toric existence}
Let $S$ be a compact toric Sasaki manifold with positive basic first Chern class with decomposition 
satisfying \eqref{Sdecomp}.  Then $S$ admits coupled Sasaki-Einstein metrics for the decomposition \eqref{Sdecomp}
if and only if $\Fut$ identically vanishes.
\end{theorem}
\noindent
Here a Sasaki manifold is said to be toric if the the cone $C(S)$ is toric, that is, if $C(S)$ admits an effective $(\bfC^\ast)^{m+1}$-action. Theorem \ref{toric existence} follows from an existence result of toric coupled Sasaki-Ricci solitons, see Theorem \ref{toricsoliton} in section 5.
In the case of toric Sasaki-Einstein metrics one can deform the Reeb vector field so that $\Fut$ vanishes, and
proving the existence of Sasaki-Ricci solitons one can conclude that there always exists a Sasaki-Einstein metric
under the condition of $c_1(D) = 0$, see \cite{MSY2}, \cite{futakionowang}, \cite{FO_ICCM_Notices19}. It is not clear whether such a volume minimization argument applies in the coupled case.

This paper is organized as follows. In section 2 we review the transverse \ka structure and
the notions of transverse holomorphic vector fields and Hamiltonian holomorphic vector fields. 
In section 3 the proof of Theorem \ref{thm2} is given. We also prove in Theorem \ref{isom} an identification of the Lie algebra of
all transverse holomorphic vector fields with the Lie algebra expressed in terms  Hamiltonian functions 
satisfying a normalization condition.
In section 4 the proof of Theorem \ref{Futaki inv} is given. We also give 
an obstruction to the existence of coupled Sasaki-Ricci solitons. 
In section 5 we first show in Theorem \ref{Minkowski} that the normalization in Theorem \ref{isom} is equivalent to a Minkowski sum condition of the moment map image. Using this we reduce the proof of Theorem \ref{toric existence}
to the same type of real Monge-Amp\`ere equations as considered by Hultgren \cite{Hultgren17}. In section 6 we supplement the proof of Theorem \ref{toric existence} by making 
the standard moment map for ample anti-canonical class explicit in terms of the first non-zero eigenfunctions
of the twisted Laplacian \eqref{moment1}.

\section{Transverse \ka structure on a Sasaki manifold.}
A compact Riemannian manifold $(S, g)$ of dimension $2m+1$ is called a Sasaki manifold
if its Riemannian cone  $\big(C(S), dr^2+r^2g\big)$ is a K\"ahler manifold. $S$ is identified with the submanifold $\{r=1\}$.
Using the convention $d^c = \sqrt{-1}(\barpartial - \partial)$, 
the restriction  $\eta$ of $d^c\log r$ to $S = \{r=1\}$ is a contact form. 
If we denote by $J$ the complex structure of $C(S)$, the restriction $\xi$ of the vector field $J r \frac{\partial}{\partial r}$ to
$S$ is the Reeb vector field of the contact form $\eta$ so that $i_\xi\eta = 1$ and $i_\xi d\eta = 0$.

Since the Sasaki structure is characterized by the K\"ahler structure of the cone $C(S)$ the geometry of Sasaki manifold $S$ is often described in terms of the K\"ahler geometry of $C(S)$. Therefore it is convenient to extend the Reeb vector field $\xi$ and the contact form $\eta$ on $S$ to $C(S)$, and we use the same letters to denote them. Thus on $C(S)$ we have
\begin{equation}\label{RCcone}
\xi = J r \frac{\partial}{\partial r}\ \mathrm{and}\ \eta = d^c\log r.
\end{equation}
As this shows, when the holomorphic structure of the cone $C(S)$ is fixed, the radial function $r$ has all the information about the Sasaki structure on $S$ and the K\"ahler structure on $C(S)$, and in fact the K\"ahler form on $C(S)$ is given by $\frac{\sqrt{-1}}{2} \partial\barpartial r^2$. 
The complex vector field $\xi-\sqrt{-1}J\xi$  is a holomorphic vector field on $C(S)$. It generates a $\bf C^*$ action on $C(S)$. The local orbit of this action defines a transversely holomorphic foliation on $S$,  given by  one dimensional leaves generated by $\xi$. 
Let $\cup_\gl U_\gl$ be an open covering of $S$ with $\pi_\gl:U_\gl\to V_\gl\subset\bfC^m$ a submersion to the local orbit spaces. Then, when  $U_\gl\cap U_\gm\neq\varnothing$, 
\[\pi_\gl\circ\pi_\gm^{-1}:\pi_\gm(U_\gl\cap U_\gm)\to \pi_\gl(U_\gl\cap U_\gm)\]
is a biholomorphic transformation. We have the transverse \ka structure in the following sense. 
On each $V_\gl$, we can give a \ka structure as follows. Let $D=\ker\eta\subset TS$, i.e. 
\[D=\{X\in TS\big|\ \eta(X)=0\}.\]
There is a canonical isomorphism 
\[d\pi_\gl: D_p\to T_{\pi_\gl(p)} V_\gl, \quad (p\in U_\gl).\]
Then there is a K\"ahler form $\gamma_\gl$ on $V_\gl$ such that 
$\pi_\gl^\ast\gamma_\gl=\frac{1}{2}d\eta$.

Let $\cup_\gl U_\gl$ with $U_\gl=\{(x_\gl, z_\gl^1,\dots, z_\gl^m)\}$ be the foliation chart on $S$. If $U_\gl\cap U_\gm\neq\varnothing$, $(x_\gl, z_\gl^1,\dots, z_\gl^m)$ and $(x_\gm, z_\gm^1,\dots, z_\gm^m)$ are local foliation coordinates on $U_\gl$ and $U_\gm$ respectively, where $\frac{\pa}{\pa x_\gl}=\xi |_{U_\gl}$. Then 
\[\frac{\pa z_\gm^i}{\pa x_\gl}=0, \quad \frac{\pa z_\gm^i}{\pa\bar z^j_\gl}=0.\]
Consequently, a $(p, q)$-form $\ga=\ga_{i_1 \dots, i_p,\bar j_1\dots \bar j_q}dz^{i_1}\wedge\dots\wedge dz^{i_p}\wedge d\bar z^{j_1}\wedge\dots\wedge d\bar z^{j_q}$ is well-defined on $S$. 

\begin{definition}
A $p$-form $\ga$ on $S$ is called basic if 
\[i_\xi\ga=0,\quad L_\xi\ga=0.\]
Let $\Lambda_B^p$ be the sheaf of germs of basic $p$-forms, and $\Omega_B^p$ be the set of all global sections of $\Lambda_B^p$. 
\end{definition}

The lifted K\"ahler form $\pi_\gl^\ast \gamma_\gl=\frac{1}{2}d\eta$ is a basic $(1, 1)$-form, in local holomorphic foliation coordinate $(x, z^1, \dots, z^m)$, we write 
\[\frac12d\eta=\sqrt{-1}g_{i\bar j}^Tdz^i\wedge d\bar z^j.\]

For basic forms, we have the following two lemmas which can be proved using Stokes theorem and the fact that $d\eta$ is basic.
\begin{lemma}\label{basic1}
If $\ga$ is a basic $(2m-1)$-form, then 
\[\int_S d_B\ga\wedge\eta=0.\]
\end{lemma}
\begin{lemma}\label{basic2}
If $\ga$, $\gb$ are basic forms with $\deg\ga+\deg\gb=2m-1$, then
\[\int_S d_B\ga\wedge\gb\wedge\eta=(-1)^{\deg\ga}\int_S \ga\wedge d_B\gb\wedge\eta.\]
\end{lemma}

We also have well-defined operators 
\[\pa_B: \Lambda_B^{p, q}\to \Lambda_B^{p+1, q}, \quad \bar\pa_B:\Lambda_B^{p, q}\to \Lambda_B^{p, q+1}.\]
Put 
\[d_B=\pa_B+\bar\pa_B, \quad d_B^c=\sqrt{-1}(\bar\pa_B-\pa_B),\]
then 
\[d_B=d|_{\Omega_B^p},\quad d_Bd_B^c=2\sqrt{-1}\pa_B\bar\pa_B,\quad d_B^2=(d_B^c)^2=0.\]
The basic $p$-th de Rham cohomology group is
\[H^p_B(S)=\frac{\ker\{ d_B:\Omega^p_B\to \Omega^{p+1}_B\}}{Im\{d_B:\Omega_B^{p-1}\to\Omega^p_B \}},\]
and basic $(p, q)$-Dolbeault cohomology group is
\[H_B^{p,q}(S)=\frac{\ker\{ \bar\pa_B:\Omega^{p, q}_B\to \Omega^{p, q+1}_B\}}{Im\{\bar\pa_B:\Omega_B^{p, q-1}\to\Omega^{p, q}_B \}}.\]
On $S$. With respect to the volume form $(\frac{1}{2}d\eta)^m\wedge\eta$, we define the adjoint operators of $d_B^*:\Omega_B^{p+1}\to \Omega_B^p$ by
\[<d_B^*\ga, \gb>=<\ga, d_B\gb>.\]
Similarly, the adjoint operators of $\bar\pa_B^*:\Omega_B^{p, q+1}\to \Omega_B^{p, q}$ is defined by 
\[<\bar\pa_B^*\ga, \gb>=<\ga, \bar\pa_B\gb>.\]
The corresponding basic Laplacian operators are defined by 
\[\Box_{d_B}=d_Bd_B^*+d_B^*d_B, \quad \Box_{\bar\pa_B}=\bar\pa_B\bar\pa_B^*+\bar\pa_B^*\bar\pa_B.\]
It is known that by the transverse \ka structures, $\Box_{d_B}=2\Box_{\bar\pa_B}$.

The transverse Ricci form $\Ric^T$ is defined as
\begin{equation}\label{TRic}
\Ric^T=-\sqrt{-1}\pa_B\bar\pa_B\log\det (g^T).
\end{equation}
$\Ric^T$ is a $d_B$-closed form and defines the basic cohomology class of type $(1, 1)$. 
The basic cohomology class $c_1^B(S)=[\Ric^T]/2\pi$ is called the basic first Chern class of $S$. We say the basic first Chern class $c_1^B(S)$ is positive, if $c_1^B(S)$ is represented by a positive basic $(1,1)$-form.
\begin{remark}
Note that if two real closed basic $(1,1)$-forms $\sigma_1$ and $\sigma_2$ represent the same basic cohomology class
there is a basic smooth function $\varphi$ such that $\sigma_1 - \sigma_2 = \sqrt{-1}\partial_B\barpartial_B \varphi$.
If we fix the complex structure of the K\"ahler cone $C(S)$ and the Reeb vector field $\xi$, the Sasaki structure can be
deformed by the change of the radial function $r$ by $\tilde r = e^\varphi r$ for a basic function $\varphi$. 
Then the contact form is deformed from $\eta = d^c \log r$ to $\tilde \eta = \eta + d^c \log \varphi$ and thus the transverse K\"ahler form is deformed from $\frac12 d\eta$ to $\frac12 d\eta + \sqrt{-1}\partial\barpartial \varphi$. The basic Chern class
is independent of choice of such contact form $\eta$. 
\end{remark}

\begin{definition}\label{t-holo.v.f.}
Let $S$ be a compact Sasaki manifold, $\xi$ the Reeb vector field and $\mathfrak{h}(S)$ the Lie algebra of all holomorphic vector fields on the cone $C(S)$. We define
\[\mathfrak{h}^T(S):=\{X\in \mathfrak{h}(S)|\ [\xi, X]=0\}\]
to be the Lie algebra of transverse holomorphic vector fields.
\end{definition}
We remark that for $X\in \mathfrak{h}^T(S)$, we also have $[\xi-\sqrt{-1}J\xi, X]=0$. It follows that $X$ descends to a holomorphic vector field on $S$, and also that $X$ descends to local orbit spaces of the Reeb flow. By abuse of notation, 
we use the same letter $X$ to denote the corresponding vector fields on $S$ and local orbit spaces of the Reeb flow.

\begin{definition}\label{Hamhol}
A complex vector field $X$ on a Sasaki manifold is called a Hamiltonian holomorphic vector field if
\begin{enumerate}
\item $d\pi_\ga(X)$ is a holomorphic vector field on $V_\ga$.
\item The complex valued function $u_X=\sqrt{-1}\eta(X)$ satisfies
\[\bar\pa_B u_X=-\sqrt{-1}i_X(\frac12 d\eta).\]
\end{enumerate} 
\end{definition}

Suppose that the basic first Chern class is positive then for any other basic K\"ahler class there is a
basic K\"ahler form of positive transverse Ricci form and thus the basic first cohomology vanishes by the standard Bochner technique. 
From this we have the following lemma.
\begin{lemma}[\cite{cho-futaki-ono}]\label{cfo}
Let $S$ be a compact Sasaki manifold of positive basic first Chern class. Then the Lie algebra of the
automorphism group of transverse holomorphic structure is the Lie algebra of all
Hamiltonian holomorphic vector fields.
\end{lemma}




Let $\gamma$ be a basic $(1,1)$-class which contains a positive $(1,1)$-form $\omega$. Then we say that $\omega$
is a basic K\"ahler form and that $\gamma$ is a basic K\"ahler class. For example the transverse K\"ahler form 
$\frac12 d\eta$ is a basic K\"ahler form and its basic cohomology class $[\frac12 d\eta]$ is a basic K\"ahler class.
If we assume the basic first Chern class $c_1^B(S)$ is positive then it is a basic K\"ahler class. 
For each basic K\"ahler form $\omega$ we can define the transverse Ricci form $\Ric^T(\omega)$ as in \eqref{TRic},
and it represents $c_1^B(S)$. Definition \ref{Hamhol} and Lemma \ref{cfo} also apply even if 
we replace $\frac12 d\eta$ by $\omega$.

 \section{Transverse holomorphic vector fields and transverse elliptic operators}

Let $(S, g)$ be a compact Sasaki manifold. We assume the basic first Chern class $c_1^B(S)>0$, and it admits a basic decomposition $c_1^B(S)=\sum\limits_{\ga=1}^k \gamma_\ga/2\pi$, each $\gamma_\ga$ is a basic \ka class. We fix a basic \ka form $\omega_\ga$ in each basic \ka class $\gamma_\ga$.  In this section, we prove Theorem \ref{thm2} stated in the Introduction
which is the relationship between the Lie algebra of transverse holomorphic vector fields and the twisted basic Laplacians on $S$.

For each basic \ka form $\omega_\ga$, since both $\Ric^T(\omega_\ga)$ and $\sum\limits_{\ga=1}^k\omega_\ga$ are in the basic first Chern class $c_1^B(S)$, by the transverse $\pa_B\bar\pa_B$ lemma \cite{el}, there is a basic function $f_\ga$ satisfying \eqref{potential}.
We normalize 
$f_\ga$
 by \eqref{normalize0}. This means $e^{f_\ga} \omega_\ga^m$ is independent of $\ga$. 
If we put 
\begin{equation}\label{volform}
dV := e^{f_\ga} \omega_\ga^m \wedge \eta
\end{equation}
then $dV$ defines a volume form on $S$. We may normalize so that $\int_S dV = 1$.


For a transverse holomorphic vector field $V\in\mathfrak{h}^T(S)$, by Lemma \ref{cfo}, let the basic function $u_\ga$ be the Hamiltonian function of $V$ with respect to the basic \ka form $\omega_\ga$.
In local holomorphic foliation coordinate 
$(x, z^1, \dots, z^m)$,  $\omega_\ga=\sqrt{-1}g^\ga_{i\bar j}dz^i\wedge d\bar z^j$. Since for each $\ga$ from $1$ to $k$, $\bar\pa_B u_\ga=\sqrt{-1}i_V\omega_\ga$, 
we have the coordinate component of $V$ are given by
\begin{equation}\label{63}
V=\grad_\ga u_\ga=\grad_\gb u_\gb.
\end{equation}
Define $\Delta_{\ga, f_\ga}$ by \eqref{t.b.Lap}.
\begin{proposition}\label{prop1}
For a transverse holomorphic vector field $V$,  the basic functions $u_1,\ \cdots,\ u_k$ satisfying \eqref{63} 
after suitable modifications by addition of constants satisfy
\begin{equation}\label{64}
\Delta_{1, f_1}u_1=\Delta_{2, f_2} u_2=\cdots=\Delta_{k, f_k} u_k=-\sum_{\gb=1}^k u_\gb.
\end{equation}
\end{proposition}

\begin{proof}
First, we show that for $\ga\neq \gb$, 
\[\Delta_{\ga, f_\ga} u_\ga=\Delta_{\gb, f_\gb} u_\gb.\]
Note that (\ref{63}) implies 
$$\pa_{\bar p} u_\ga=g^\ga_{i\bar p}g_\gb^{i\bar j}\pa_{\bar j} u_\gb.$$ 
Thus
\begin{equation*}
\pa_k\pa_{\bar p} u_\ga=\pa_k\big(g_{i\bar p}^\ga g_\gb^{i\bar j}\pa_{\bar j} u_\gb\big)
= g_{i\bar p}^\ga g_\gb^{i\bar j}\pa_k\pa_{\bar j} u_\gb+g_\gb^{i\bar j}\pa_{\bar j} u_\gb\pa_k g^\ga_{i\bar p} + g_{i\bar p}^\ga\pa_k g_\gb^{i\bar j} \pa_{\bar j} u_\gb.
\end{equation*}
Taking trace with respect to $g_\ga^{k\bar p}$ and using the K\"ahler condition $\partial_i g^\ga_{j\bark} 
= \partial_j g^\ga_{i\bark}$, 
\begin{align*}
\Delta_\ga u_\ga&=\Delta_\gb u_\gb+ \big(g_\ga^{k\bar p}\pa_k g_{i\bar p}^\ga\big)g_\gb^{i\bar j} \pa_{\bar j} u_\gb+\pa_i g_\gb^{i\bar j}\pa_{\bar j} u_\gb\\
&=\Delta_\gb u_\gb+ g_\gb^{i\bar j} (\pa_i\log \omega^m_\ga)\pa_{\bar j} u_\gb- g_\gb^{i\bar j}(\pa_i\log \omega^m_\gb)\pa_{\bar j} u_\gb\\
&=\Delta_\gb u_\gb+ g_\gb^{i\bar j} \pa_i\log\frac{\omega_\ga^m}{\omega_\gb^m}\pa_{\bar j} u_\gb\\
&=\Delta_\gb u_\gb+ g_\gb^{i\bar j} \pa_i f_\gb\pa_{\bar j} u_\gb-g_\gb^{i\bar j} \pa_i f_\ga\pa_{\bar j} u_\gb\qquad\qquad (\star 1)\\
&=\Delta_\gb u_\gb+ g_\gb^{i\bar j} \pa_i f_\gb\pa_{\bar j} u_\gb-g_\ga^{i\bar j} \pa_i f_\ga\pa_{\bar j} u_\ga\qquad\qquad (\star 2).
\end{align*}
where 
$(\star 1)$  is due to $e^{f_\ga}\omega_\ga^m=e^{f_\gb}\omega_\gb^m$, and $(\star 2)$ is due to 
$\grad_{\ga} u_\ga=\grad_{\gb} u_\gb$. Hence
\[\Delta_{\ga, f_\ga} u_\ga=\Delta_{\gb, f_\gb} u_\gb.\]

Next, we show for any $\ga=1, \dots, k$, after a normalization, 
\[\Delta_{\ga, f_\ga} u_\ga+\sum\limits_{\gb=1}^k u_\gb=0.\]
We denote by $\nabla_\ga$ the covariant derivative with respect to $\omega_\ga$.
Then for smooth functions $\nabla_\ga^{\prime\prime} = \barpartial_B$ for any $\ga$, and for this reason the derivative
with respect to $i$-th coordinate will be written as $\nabla_{\ga, \bari} = \nabla_{B, \bari}$. 
Using the fact that for all $\ga$, $\nabla_{\ga, \bar i}\nabla_{\ga, \bar j} u_\ga=0$, and 
$\nabla^i_{\ga} u_\ga=\nabla^i_{\gb} u_\gb$, we have
\begin{align*}
\nabla_{B,\bark}(\Delta_{\ga, f_\ga} u_\ga)&=\nabla_{\ga, \bar k}\nabla_{\ga, i}\nabla_{\ga}^i u_\ga
+\nabla_{\ga}^i u_\ga\nabla_{\ga, \bark}\nabla_{\ga, i} f_\ga\\
&=-g^{i\bar j}_\ga R_{\ga, i\bar k}\nabla_{\ga, \barj} u_\ga
+g^{i\bar j}_\ga\nabla_{\ga, \bar k}\nabla_{\ga, i} f_\ga\nabla_{\ga, \bar j} u_\ga\\
&=\big(-\sum_{\gb=1}^k g_{i\bar k}^\gb\big)\nabla_{\ga}^i u_\ga
=-\sum_{\gb=1}^k g_{i\bar k}^\gb\nabla_{\gb}^i u_\gb\\
&=-\nabla_{B, \bar k}\big(\sum_{\gb=1}^k  u_\gb\big).
\end{align*}
It follows from the assumptions that $u_\ga$ are basic functions we obtain
$$\Delta_{\ga, f_\ga} u_\ga+\sum_{\gb=1}^k u_\gb=c_\ga$$
 for constants $c_\ga.$
 In particular, $c_\ga=c_\gb=c$. Furthermore, we choose
$v_\ga=u_\ga-\frac{c}{k}$, then $v_\ga$ satisfies
$$\Delta_{\ga, f_\ga} v_\ga+\sum\limits_{\gb=1}^k v_\gb=0.$$
\end{proof}
\begin{theorem}\label{isom}
Let $S$ be a compact Sasaki manifold with $c_1^B(S) > 0$, the decomposition \eqref{Sdecomp} and the choice of 
basic K\"ahler forms $\omega_\ga$ and basic smooth functions $f_\ga$ satisfying \eqref{potential}. Then the Lie algebra 
$\mathfrak h^T(S)$ of all transverse holomorphic vector fields is isomorphic to 
$$ \mathfrak h^\prime := \{u_1+ \cdots + u_k\ |\ \grad_\ga u_\ga=\grad_\gb u_\gb \in \mathfrak h^T(S),\ \ga, \gb=1, \dots, k, 
\int_S (u_1 + \cdots + u_k) dV = 0 \}$$
where the Lie algebra structure of the latter is given by the Poisson bracket of each $u_\ga$.
Here the isomorphism is given by $V \mapsto u_1+ \cdots + u_k$ with $V = \grad_\ga u_\ga$. 
\end{theorem}
\begin{proof} 
There is a natural injection from $\mathfrak h^\prime$ to $\mathfrak h^T(S)$ sending $u_1 + \cdots + u_k$ to
$V = \grad_\ga u_\ga$. This map is also surjective 
by Proposition \ref{prop1}.
\end{proof}
Now we turn to the proof of Theorem \ref{thm2}. First we prove it in the regular Sasaki case, namely the K\"ahler case.
The statement in this case should be helpful to understand the product configuration of the definition of K-stablility in the coupled K\"ahler case. 
\begin{theorem}\label{kahlercase}
Let $M$ be a Fano \ka manifold of complex dimension $n$. Suppose that for \ka forms $\omega_1, \dots, \omega_k$, there exist real smooth functions $f_1, \dots, f_k\in C^\infty(M)$, such that
\begin{align*}
\begin{cases}
\Ric(\omega_\ga)=\sum\limits_{\gb=1}^k\omega_\gb+\sqrt{-1}\pa\bar\pa f_\ga,\\
e^{f_1}\omega_1^m=e^{f_2}\omega_2^m=\cdots=e^{f_k}\omega_k^m.
\end{cases}
\end{align*}
Suppose also that non-constant complex valued smooth functions $u_1, \dots, u_k\in C^\infty(X)\otimes \bf C$ satisfy
\begin{enumerate}
\item $\nabla_\ga^i u_\ga=\nabla_\gb^i u_\gb$ for  $i=1, 2, \dots, n$ and $\ga, \gb=1, 2, \dots, k$.
\item $-\Delta_{\ga, f_\ga} u_\ga=\lambda\sum\limits_{\gb=1}^k u_\gb$ for $\ga=1, 2, \dots, k$, where $\Delta_{\ga, f_\ga} u_\ga=\Delta_\ga u_\ga+\nabla_\ga^i u_\ga\nabla_i f_\ga$, and $\Delta_\ga=-\bar\pa_\ga^*\bar\pa$ is the Beltrami-Laplacian with respect to the \ka form $\omega_\ga$.
\end{enumerate}
Then $\lambda\ge 1$. Moreover, if $\lambda=1$, the vector field $V=\grad_\ga u_\ga=\grad_\gb u_\gb$ is a holomorphic vector field.
\end{theorem}
\begin{proof}
We compute
\begin{eqnarray}\label{positive}
&&\int_M |\nabla^{\prime\prime}_\ga\nabla^{\prime\prime}_\ga u_\ga|_{\omega_\ga}^2\ e^{f_\ga}\omega_\ga^m \\
&=&\int_M \nabla^{i}_\ga \nabla^{j}_\ga u_\ga \nabla_{\ga,i}\nabla_{\ga,j} \bar u_\ga\ e^{f_\ga}\omega_\ga^m \nonumber\\
&=&-\int_M\nabla_{\ga, j}\big(\nabla_\ga^{i}\nabla_\ga^{j} u_\ga e^{f_\ga}\big)
\nabla_{\ga,i}\bar u_\ga\ \omega_\ga^m \nonumber\\
&=&-\int_M \Big(\nabla_\ga^{i}(\Delta_\ga u_\ga)+R_{\ga,}{}_j{}^i\nabla_\ga^{j} u_\ga+\nabla_\ga^{i}\nabla_\ga^{j} u_\ga\nabla_j^\ga f_\ga\Big)\nabla_{\ga,i}\bar{u}_\ga\ e^{f_\ga}\omega_\ga^m \nonumber\\
&=&-\int_M \Big(\nabla_\ga^{i}(\Delta_\ga u_\ga)+\nabla_{\ga, j} \nabla_\ga^{i} f_\ga\nabla_\ga^{j} u_\ga+\nabla_\ga^{i}
\nabla_\ga^{j} u_\ga \nabla_{\ga, j} f_\ga\Big)\nabla_{\ga, i}\bar{u}_\ga\ e^{f_\ga}\omega_\ga^m \nonumber\\
&&-\int_M \big(\sum\limits_{\gb=1}^k g_{\gb,i\bar j}\big)\nabla^i_\ga u_\ga\bar{\nabla_\ga^j u_\ga}\ e^{f_\ga}\omega_\ga^m\nonumber\\
&=&-\int_M \nabla_\ga^{i}(\Delta_{\ga, f_\ga} u_\ga)\nabla_{\ga,i}\bar u_\ga\ e^{f_\ga}\omega_\ga^m - \int_M 
\sum\limits_{\gb=1}^k g_{\gb,i\bar j}\nabla^i_\gb u_\gb\bar{\nabla_\gb^j u_\gb}\ e^{f_\gb}\omega_\gb^m\hspace{1 in} \nonumber\\
&=&\int_M |\Delta_{\ga,f_\ga} u_\ga|^2\  e^{f_\ga}\omega_\ga^m+\sum\limits_{\gb=1}^k\int_M \big(\Delta_{\gb,f_\gb} u_\gb\big)\bar u_\gb \ e^{f_\gb}\omega_\gb^m\nonumber\\
&=&\int_M \lambda^2|\sum\limits_{\gb=1}^k u_\gb|^2\  e^{f_\ga}\omega_\ga^m-\lambda \sum\limits_{\gb=1}^k\int_M \big(\sum\limits_{\gamma=1}^k u_\gamma\big)\bar u_\gb\ e^{f_\gb}\omega_\gb^m\nonumber\\
&=&\lambda(\lambda-1)\int_M |\sum\limits_{\gb=1}^k u_\gb|^2\ e^{f_\gb}\omega_\gb^m, \nonumber
\end{eqnarray}
where we have used $e^{f_\ga}\omega_\ga^m=e^{f_\gb}\omega_\gb^m$ and $\nabla^i_\ga u_\ga=\nabla^i_\gb u_\gb$.
Taking the $L^2$-inner product with $u_\gamma$ on both sides of 
$-\Delta_{\ga, f_\ga} u_\ga=\lambda\sum\limits_{\gb=1}^k u_\gb$ 
and taking sum over $\gamma = 1, \cdots, k$, we see that $\lambda > 0$ since we assumed $u_\ga$ are non-constants.
Then from the computations \eqref{positive} we conclude $\lambda \ge 1$. Moreover, if $\lambda=1$, then $V=\grad_\ga u_\ga$ is a holomorphic vector field.
\end{proof}


\begin{proof}[Proof of Theorem \ref{thm2}]
In the proof of Theorem \ref{kahlercase} we replace the volume form $e^{f_1}\omega_1^m = \cdots = e^{f_k}\omega_k^m$
by $dV = e^{f_1}\omega_1^m \wedge \eta = \cdots = e^{f_k}\omega_k^m \wedge \eta$. 
Then by Lemma \ref{basic1} and Lemma \ref{basic2} the same computations of Theorem \ref{kahlercase} proves (1) of Theorem \ref{thm2}. The part (2) follows from Lemma \ref{cfo}, Proposition \ref{prop1} and Theorem \ref{kahlercase}.
\end{proof}

\section{The invariant for coupled Sasaki-Einstein manifold}

Let $S$ be a compact Sasakian manifold, $\xi$ be the Reeb vector field and $\eta$ be the contact form. We assume that the basic first Chern class $c_1^B(S)$ is positive and that it admits a basic decomposition $c_1^B(S)=(\sum\limits_{\ga=1}^k\gamma_\ga)/2\pi$, where $\gamma_1, \dots, \gamma_k$ are basic \ka classes. In this section, we define an invariant on $S$ which gives an obstruction to the existence of coupled Sasaki-Einstein metrics. This invariant extends the obstruction to the existence of \ke metrics obtained in \cite{futaki83.1}, but it is expressed in the form obtained in Proposition 2.3 \cite{futakimorita85}.

As in the Introduction, we let $\omega_\ga$ be the basic K\"ahler forms in $\gamma_\ga$ and $f_\ga$ be the basic smooth functions such that $\Ric^T(\omega_\ga)=\sum\limits_{\gb=1}^k\omega_\gb+\sqrt{-1}\pa_B\bar\pa_B f_\ga$, and the basic functions $f_1, \dots, f_k$ are normalized by (\ref{normalize0}), i.e.
\[e^{f_1}\omega_1^m=\cdots=e^{f_k}\omega_k^m.\]
Taking a transverse holomorphic vector field $V\in \mathfrak{h}^T(S)$, let complex-valued basic functions $u_1, \cdots,  u_k$ be the Hamiltonian functions of $V$ with respect to basic \ka forms $\omega_1, \cdots, \omega_k$.

In the Introduction we defined $\Fut: \mathfrak{h}^T(S) \to \bfC$ by \eqref{futaki} using Theorem \ref{thm2}, (2). 
But since the condition \eqref{64} is equivalent to the normalization $\int_S (u_1 + \cdots + u_k)\  dV = 0$
we may re-define $\Fut$ using the latter normalization. Then 
this definition is independent of the choice $u_1,\ \dots,\ u_k$ satisfying
the normalization $\int_S (u_1 + \cdots + u_k)\ dV = 0$
since if $\tu_1,\ \cdots,\ \tu_k$ are another choice to represent the same element in $\mathfrak h^\prime$ then $\tu_\ga = u_\ga + c_\ga$
for some constants $c_1,\ \cdots, c_k$ with $c_1 + \cdots + c_k = 0$. 
\begin{proof}[Proof of Theorem \ref{Futaki inv}]
If we take basic K\"ahler forms $\tilde\omega_\ga=\omega_\ga+\sqrt{-1}\pa_B\bar\pa_B\varphi_\ga$, where $\varphi_1, \dots, \varphi_k$ are basic functions. Then 
$\sum\limits_{\ga=1}^k\tilde\omega_\ga=\sum\limits_{\ga=1}^k\omega_\ga+\sqrt{-1}\pa_B\bar\pa_B\big(\sum\limits_{\ga=1}^k\varphi_\ga\big).$
Let $\tilde f_\ga$ be the basic functions satisfying 
\begin{align*}
\begin{cases}
\Ric^T(\tilde\omega_\ga)=\sum\limits_{\gb=1}^k\tilde\omega_\gb+\sqrt{-1}\pa_B\bar\pa_B\tilde f_\ga\\
e^{\tilde f_1}\tilde\omega_1^m=\dots=e^{\tilde f_k}\tilde\omega_k^m.
\end{cases}
\end{align*}
Then
\begin{equation}
e^{\sum\limits_{\gb=1}^k\varphi_\gb}e^{\tilde f_\ga}\tilde\omega_\ga^m=e^{f_\ga}\omega_\ga^m.
\end{equation}

For the transverse holomorphic vector field $V$, we take the Hamiltonian functions $u_\ga$ with respect to
$\omega_\ga$ satisfying the
normalization condition $\int_S (u_1 + \cdots + u_k)\ dV = 0$ or equivalently the condition \eqref{64}.
We then take $\tilde u_\ga=u_\ga+V(\varphi_\ga)$ to be  the Hamiltonian function of $V$ with respect to
$\tilde\omega_\ga$. 
We first show that $\tilde u_\ga$'s satisfy the normalization condition $\int_S (\tilde u_1 + \cdots + \tilde u_k)\ d\tilde V = 0$
where $d\tilde V$ denotes the corresponding volume form $e^{\tilde f_\ga}\tilde\omega_\ga^m\wedge\eta$. We have 
$$\sum\limits_{\ga=1}^k\tilde u_\ga=\sum\limits_{\ga=1}^k u_\ga+V\big(\sum\limits_{\ga=1}^k\varphi_\ga\big).$$
We compute using the condition \eqref{64} that
\begin{align*}
&\int_S (\tilde u_1 + \cdots + \tilde u_k)\ d\tilde V
=\int_S\big(\sum\limits_{\gb=1}^k \tilde u_\gb\big)\ e^{\tilde f_\ga}\tilde\omega_\ga^m\wedge\eta\\
=&\int_S\Big(\sum\limits_{\ga=1}^k u_\ga+V\big(\sum\limits_{\gb=1}^k\varphi_\gb\big)\Big)\ e^{-\sum\limits_{\gb=1}^k\varphi_\gb}e^{f_\ga}\omega_\ga^m\wedge\eta\\
=&-\int_S\Delta_{\ga, f_\ga} u_\ga\ e^{-\sum\limits_{\gb=1}^k\varphi_\gb}e^{f_\ga}\ \omega_\ga^m\wedge\eta-\int_SV\Big(e^{-\sum\limits_{\gb=1}^k\varphi_\gb}\Big)e^{f_\ga}\omega_\ga^m\wedge\eta\\=&0.
\end{align*}

Next we show that $\mathrm{Fut}$ is independent of the choice of the basic K\"ahler forms $\omega_\ga$ in $\gamma_\ga$.
To show this, consider the deformation 
$\omega_\ga(t)=\omega_\ga+\sqrt{-1}\pa_B\bar\pa_B(t\varphi_\ga)$ 
for basic functions $\varphi_\ga$. Let $u_\ga$ and $u_\ga(t)$ be the Hamiltonian functions of $V$ with respect to $\omega_\ga$ and $\omega_\ga(t)$, where $u_\ga(t)=u_\ga+ tV(\varphi_\ga)$. Then it is enough to show the integrals
$\int_S u_\ga(t)\ \big(\omega_\ga(t)\big)^m\wedge\eta$ and $\int_S\big( \omega_\ga(t)\big)^m\wedge\eta$ are 
independent of $t$.
But this follows from
\begin{align*}
 &\frac{d}{dt}\int_S\big( \omega_\ga(t)\big)^m\wedge\eta=\int_S (\Delta_{\omega_{\ga}(t)}\varphi_\ga)\ (\omega_{\ga}(t))^m\wedge\eta=0, \text{ and }\\
&\frac{d}{dt}\int_S u_\ga(t)\ \big(\omega_\ga(t)\big)^m\wedge\eta=\int_S V(\varphi_\ga)\big(\omega_\ga(t)\big)^m\wedge\eta+\int_S u_\ga(t)\Delta_{\omega_\ga(t)}\varphi_\ga\ \big(\omega_\ga(t)\big)^m\wedge\eta=0.
\end{align*}

If $S$ admits coupled Sasaki-Einstein metric $\omega_1, \cdots, \omega_k$, then we can take all $f_\ga$ to be zeros. Thus 
the normalization $\int_S (u_1 + \cdots + u_k)\ dV = 0$ implies $\Fut(V)=0$.
This completes the proof.
\end{proof}

Let $\omega_\alpha$ be a basic K\"ahler form in the basic K\"ahler class $\gamma_\ga$.
For $V,\ W_1,\ \cdots, W_k \in \mathfrak h^T(S)$ we put
\begin{equation}\label{TZ2}
\Fut_{W_1, \cdots, W_k}(V) = \sum_{\ga=1}^k \frac{\int_S u_{\ga,V}\ e^{u_{W_\ga}} \omega_\ga^m \wedge \eta}{\int_S e^{u_{W\ga}} \omega_\ga^m \wedge \eta}
\end{equation}
where $\grad_\ga u_{\ga,V} = V$ and $\grad_\ga u_{W_\ga} = W_\ga$ and we assume $u_{\ga,V}$'s
satisfy \eqref{64} or equivalently 
the normalization 
$$\int_S (u_{1,V} + \cdots + u_{k,V})\ dV = 0.$$
By a similar proof as that of Theorem \ref{Futaki inv}, one can show that 
this is also independent of the choice of $\omega_\ga$'s in $\gamma_\ga$'s. When $W_1=\cdots=W_k=0$ , 
$\Fut_{W_1, \cdots, W_k}(V)$ coincides with $\Fut(V)$ in \eqref{futaki}.
\begin{theorem}\label{KRSobst}
If there exists coupled Sasaki-Ricci solitons \eqref{potential} for the decomposition $c_1^B(S)=\sum_{\ga=1}^k\gamma_\ga/2\pi$
then $\Fut_{W_1, \cdots, W_k}$ identically vanishes for $W_\ga = \grad_\ga f_\ga$, $\ga = 1, \cdots, k$.
\end{theorem}
\begin{proof}
Suppose we have a solution of coupled Sasaki-Ricci solitons \eqref{potential}. We may normalize $f_\ga$ so that 
\eqref{normalize0} is satisfied. Then since $W_\ga = \grad_\ga f_\ga$ and thus $u_{W_\ga} = f_\ga$ in \eqref{TZ2}, 
we obtain using \eqref{volform}
\begin{equation}\label{TZ3}
\Fut_{W_1, \cdots, W_k}(V) = \left(\sum_{\ga=1}^k \int_S u_{\ga,V} dV\right)/\int_S dV,
\end{equation}
which vanishes by Theorem \ref{isom}. This completes the proof of Theorem \ref{KRSobst}.
\end{proof}
\noindent
The above theorem is an extension of \cite{futaki83.1}, \cite{TianZhu02}.

\section{Toric Sasaki-Ricci solitons}

Let $T$ be a real torus of dimension $m+1$ acting effectively 
on $S$ as isometries and $\mathfrak t$ its Lie algebra. Naturally $T$ acts on the K\"ahler cone
$C(S)$ as holomorphic isometries.  Note that in this case we have an effective 
$(\bfC^\ast)^{m+1}$-action on $C(S)$, and $m+1$ is the maximal dimension of the torus action because of dimension reason. In such a case we say $C(S)$ is a toric K\"ahler cone and $S$ 
is a toric Sasaki manifold. See \cite{Legendre11} for a concise description of toric Sasaki geometry.

We identify an element $X \in \mathfrak t$ with a vector field on $C(S)$
and denote it by the same letter $X$. 
Note that the Reeb vector field $\xi$ lies in $\mathfrak t$. Since the K\"ahler form on $C(S)$  is given by $\frac{\sqrt{-1}}2 \partial\barpartial r^2$ the moment map 
$\mu : C(S) \to \mathfrak t^\ast$ on the K\"ahler cone $C(S)$ 
for the action of $T$ is given for $X \in \mathfrak t$ by
$$ \langle \mu,X\rangle =  r^2\eta(X)$$
where we recall that the contact form $\eta$ is extended on $C(S)$ by \eqref{RCcone}. 
It is well known that the image of $\mu$ is a rational convex polyhedral cone which we denote by 
$\mathcal C$. 
The Sasaki manifold $S \cong \{r=1\}$ is characterized as $ \langle \mu,\xi\rangle = 1$, and its moment map
image is $\{p \in \mathcal C\ |\ \langle p,\xi\rangle = 1 \}$. We denote by $\mathcal P_\xi$ the image 
of $\{p \in \mathcal C\ |\ \langle p,\xi\rangle = 1 \}$ under the projection 
$\pi_\xi : \mathfrak t^\ast \to (\mathfrak t/\bfR \xi)^\ast$. 
$\mathcal P_\xi$ is the image of $\mu_\xi := \pi_\xi \circ \mu|_S : S \to (\mathfrak t/\bfR \xi)^\ast$ which we call
the transverse moment map for $S$.
$\mathcal P_\xi$ is a rational convex polytope when $\xi$ defines a
quasi-regular Sasaki structure, but otherwise it is not rational. The inverse image by $\mu$ of 
each facet of $\mathcal C$ is the fixed point set in $C(S)$ of an one dimensional torus. If $\lambda
\in \mathfrak t$ is its infinitesimal generator then $\lambda$ is normal to the facet. The 
inverse image by $\mu_\xi$ of corresponding 
facet of $\mathcal P_\xi$ is the fixed point set in $S$ of the same one dimensional torus,
and $\lambda/\mathbf R \xi \in \mathfrak t/\mathbf R \xi$ is normal to the corresponding facet of 
$\mathcal P_\xi$. (Note that there is no meaning of rationality in $\mathfrak t/\mathbf R \xi$ 
if $\xi$ is not rational.)

Suppose that the basic first Chern class is positive. For any basic K\"ahler form $\omega$
invariant under $T$ we have a transverse moment map $\mu_\omega : S \to (\mathfrak t/\bfR \xi)^\ast$ defined up to translation by the same reason as in the paragraph after Lemma \ref{cfo}. Note that the Hamiltonian 
functions in Lemma \ref{cfo} can be taken to be real functions since $T$ acts isometries. Then  the image of the transverse moment map is a convex polytope, which we denote by $\mathcal P_\omega$, and the facets have the same
description as above. In particular, the facets of $\mathcal P_\xi$ and $\mathcal P_\omega$ are
parallel to each other.

Let $c_1^B(S)=(\sum\limits_{\ga=1}^k\gamma_\ga)/2\pi$ be a basic decomposition where $\gamma_1, \dots, \gamma_k$ are basic \ka classes, and 
choose basic K\"ahler forms $\omega_\ga \in \gamma_\ga$. 
We put $\mathcal P_\ga := \mathcal P_{\omega_\ga}$ which is independent of $\omega_\ga \in 
\gamma_\ga$ up to translation. Choose $W_1, \cdots, W_k \in \mathfrak t$. 
Then $\Fut_{W_1, \cdots, W_k}$ in \eqref{TZ2}, which is independent of choice of $\omega_\ga \in \gamma_\ga$, is expressed as 
\begin{equation}\label{TZ4}
\Fut_{W_1, \cdots, W_k} = \sum_{\ga=1}^k \mathcal A_{\mathcal P_\ga}(W_\ga)
\end{equation}
where 
$$ \mathcal A_{\mathcal P_\ga} (W_\ga) = \frac{\int_{\mathcal P_\ga} p\ e^{\langle W_\ga, p\rangle} dp}{\int_{\mathcal P_\ga} e^{\langle W_\ga, p\rangle} dp}. $$

\begin{theorem}\label{toricsoliton}
Let $S$ be a compact toric Sasaki manifold with the basic first Chern class $c_1^B(S)$
positive, and $c_1^B(S)=(\sum\limits_{\ga=1}^k\gamma_\ga)/2\pi$ be a basic decomposition.
Let $\mathcal P_\ga$ be the convex polytope which is the image of the transverse moment map 
for $\gamma_\ga$ with normalization $\sum\limits_{\ga=1}^k \int_{\mathcal P_\ga} pdp = 0$. 
Then there exist coupled K\"ahler-Ricci solitons satisfying \eqref{potential} for
$W_\ga = \grad_\ga f_\ga \in \mathfrak t$ if and only if 
$$\sum_{\ga=1}^k \mathcal A_{\mathcal P_\ga}(W_\ga) = 0.$$
\end{theorem}
Note that the condition  $\sum\limits_{\ga=1}^k \int_{\mathcal P_\ga} pdp = 0$ implies that the
barycenters of the Minkowski sum $\sum\limits_{\ga=1}^k \mathcal P_\ga$ lies at the origin. 
This condition is equivalent to the the normalization 
$$\int_S (u_{1,V} + \cdots + u_{k,V})\ dV = 0$$
in Theorem \ref{isom}. 
Let $K_S$ denote the complex line bundle over $S$ consisting of basic $(m,0)$-forms, and
call $K_S$ the transverse canonical line bundle and $K_S^{-1}$ the transverse anti-canonical
line bundle. We put $\omega := \omega_1 + \cdots + \omega_k$, which is a basic K\"ahler form
in $c_1^B(S) = c_1(K^{-1}_S)$. Let $F$ be a basic smooth function such that
\begin{equation}\label{F}
\Ric^T(\omega) - \omega = \sqrt{-1}\partial_B\barpartial_B F.
\end{equation}
We put $\Delta_F u:= \Delta u + (\grad_\omega u)F$. Then
Theorem \ref{thm2} and Theorem \ref{isom}  assert that when $k=1$ 
we have the isomorphisms of the Lie algebra $\mathfrak h^T(S)$ 
of all transverse holomorphic vector fields 
\begin{eqnarray}\label{moment1}
\mathfrak h^T(S) &\cong& \{u \in C^\infty(S)_B\otimes \bfC\ |\ -\Delta_F u = u\}\\
&=&
\{u \in C^\infty(S)_B\otimes \bfC\ |\ \int_S u\ e^F \omega^m\wedge \eta = 0\}.\nonumber
\end{eqnarray}
We will call the moment map for the class $c_1(K_S^{-1})$ defined by the Hamiltonian functions $u$
in \eqref{moment1} the standard moment map, 
and denote its moment polytope by $\mathcal P_{-K_S}$. In the Appendix we will show that this moment map is indeed 
standard.

\begin{theorem}\label{Minkowski}
The normalization condition $\int_S (u_{1,V} + \cdots + u_{k,V})\ dV = 0$ is equivalent to
$$ \sum_{\ga=1}^k \mathcal P_\ga =\mathcal P_{-K_S}$$
where the left side hand is the Minkowski sum of the polytopes $\mathcal P_\ga$'s.
\end{theorem}

\begin{proof} By \cite{yau78}, \cite{el} there exists a unique basic K\"ahler form 
$\omega_0$ in $c_1^B(S)$ such that
$\Ric^T(\omega_0) = \omega$. Then using \eqref{F} we have
$$ e^F \omega^m = \omega_0^m$$
by adding a constant to $F$.
On the other hand by \eqref{potential} with the normalization \eqref{normalize0} we also have
for any $\ga$
$$ e^{f_\ga}\omega_\ga^m = \omega_0^m.$$
Thus, using \eqref{volform}, we have for any $\ga$
$$ e^F \omega^m \wedge \eta = e^{f_\ga}\omega_\ga^m\wedge \eta = dV.$$
Hence
$$ \int_S (u_{1, V} + \cdots + u_{k, V}) dV = 0$$
implies that $u = u_{1, V} + \cdots + u_{k, V}$ satisfies the condition in \eqref{moment1}. 
This implies the Minkowski sum $\sum\limits_{\ga=1}^k \mathcal P_\ga$ coincides with 
$\mathcal P_{-K_S}$.
This completes the proof.
\end{proof}
\noindent
When $S$ is the unit circle bundle of $K_X^{-1}$ of a
Fano manifold $X$ the Minkowski sum condition in Theorem \ref{Minkowski} is equivalent to
$\sum\limits_{\ga=1}^k \mathcal P_\ga = \mathcal P_{-K_X}$ in Hultgren's paper \cite{Hultgren17}.

\begin{proof}[Proof of Theorem \ref{toricsoliton}.]
We choose basic K\"ahler forms $\omega_\ga$ in $\gamma_\ga$, and consider 
for $t \in [0,1]$ the family of Monge-Amp\`ere equations
\begin{equation}\label{MA}
e^{g_1 + W_1(\phi_1)}(\omega_1 + \sqrt{-1}\partial_B\barpartial_B \phi_1)^m =
\cdots = e^{g_k + W_k(\phi_k)}(\omega_k + \sqrt{-1}\partial_B\barpartial_B \phi_k)^m 
= e^{-t\sum\limits_\ga \phi_\ga}\omega_0^m
\end{equation}
in terms of basic functions $\phi_\ga$, 
where $\omega_0$ is the unique basic K\"ahler form 
 in a basic K\"ahler class such that
$\Ric^T(\omega_0) = \omega_1 + \cdots + \omega_k$ and $g_\ga$ is the potential
function of $W_\ga$ with respect to $\omega_\ga$, i.e. $\grad_\ga g_\ga = W_\ga$.
If we have a solution for $t=1$ the K\"ahler forms $\omega_\ga + \sqrt{-1}\partial \barpartial \phi_\ga$
give the desired coupled Sasaki-Ricci solitons. By the same argument as in \cite{pingali}, \cite{Hultgren17}
we only need to show the $C^0$-estimates. To do so, we wish to reduce the equation to
holomorphic logarithmic coordinates, further to real Monge-Amp\`ere equation with respect 
to the real coordinates, and to show the same arguments in \cite{Hultgren17} apply in our 
Sasaki situation.

As in \cite{futakionowang}, section 7, we take any subtorus $H\subset T$ of codimension 1 
such that its Lie algebra ${\mathfrak h}$
does not contain $\xi$. Let $H^c \cong ({\bf C}^*)^m$ denote
the complexification
of $H$. Take any point $p\in \mu^{-1}(\mathrm{Int}\, \mathcal C)$ and consider the orbit
$Orb_{C(S)}(H^c,p)$ of the $H^c$-action on $C(S)$ through $p$. 
Since $H^c$-action preserves $-J\xi=r\partial\slash \partial r$, it descends to an action on the set
$S \cong \{r=1\}\subset C(S)$. 
More precisely this action is described as follows. Let $\gamma : H^c \times C(S) \to
C(S)$ denote the $H^c$-action on $C(S)$.  Let ${\overline p}$ and $\overline{\gamma(g,p)}$ respectively
be the points on $\{r=1\}$ at which the flow lines through $p$ and $\gamma(g,p)$ generated by  $r\partial\slash
\partial r$ respectively meet $\{r=1\}$. 
Then the $H^c$-action on $S \cong \{r=1\}$
is given by $\bar{\gamma} : H^c \times \{r=1\} \to \{r=1\}$ where
$$\bar{\gamma}(g, {\overline p}) = \overline{\gamma(g,p)}.$$
Let $Orb_S(H^c,{\overline p})$ be the orbit of the induced action of $H^c$
on $S \cong \{r=1\}$. 
Then as in Proposition 7.2, \cite{futakionowang}, 
the transverse K\"ahler structure of the Sasaki manifold $S$ 
is completely determined by the restriction of $\frac 12 d\eta$ to $Orb_{C(S)}(H^c, p)$,
and also to $Orb_S(H^c,{\overline p})$.
For other basic K\"ahler form $\omega$ on $S$ we may restrict $\omega$ to these two
orbits, and consider the transverse K\"ahler geometry as the K\"ahler geometry on 
$Orb_{C(S)}(H^c, p)$
and $Orb_S(H^c,{\overline p})$. The two K\"ahler manifolds 
$(Orb_{C(S)}(H^c, p), \omega_{Orb_{C(S)}(H^c, p)})$ and 
$(Orb_S(H^c, \barp), \omega_{Orb_S(H^c, \barp)})$ thus obtained are essentially the same 
in that if we give them the holomorphic structures induced from the holomorphic structure
of $H^c$ then they are isometric K\"ahler manifolds. The difference between them is that
$Orb_{C(S)}(H^c, p)$  is a complex submanifold of the complex manifold $C(S)$
while $Orb_S(H^c, \barp)$ is a complex submanifold in the real Sasaki manifold $S$.
Furthermore, since the Reeb vector field can be approximated
by quasi-regular ones, we may assume that the closure of $(Orb_S(H^c, \barp), \omega_{Orb_S(H^c, \barp)})$ is a toric K\"ahler orbifold. 

For any generic point $
q' \in S$ the trajectory through $q'$ generated
 by the Reeb vector field $\xi$ meets $Orb_S(H^c, \barp)$ and
$\xi$ generates an one parameter subgroup of isometries. So, the transverse K\"ahler geometry at
any $q'$ is determined by the transverse K\"ahler geometry along the points on $Orb_S(H^c, \barp)$. This trajectory may meet $Orb_S(H^c, \barp)$ infinitely many times when the Sasaki structure is
irregular. But the transverse structures at all of them define the same
K\"ahler structure because $\xi$ generates a subtorus in $T^{m+1}$ and we assumed that $T^{m+1}$ preserves the Sasaki structure. 

Now we will express the K\"ahler potentials of $\omega_\ga$ in terms real affine coordinates
on $H^c$. 
On $Orb_{C(S)}(H^c, p) \cong  ({\bfC}^*)^m$ we use the affine logarithmic coordinates 
$$(w^1,w^2,\cdots ,w^m)=
(x^1+\sqrt {-1}\theta^1,x^2+\sqrt {-1}\theta^2,\cdots, x^m+\sqrt {-1}\theta^m)$$
for a point
$$(e^{x^1+\sqrt {-1}\theta^1}, e^{x^2+\sqrt {-1}\theta^2}, \cdots ,e^{x^m+\sqrt {-1}\theta^m}) \in 
 ({\bf C}^*)^m \cong H^c.$$

By Section A.2.3 in \cite{Guillemin}, there is a real smooth function $h_\ga$ unique up to an
affine linear function such that 
$$\omega_\ga = \sqrt{-1} \frac{\partial^2 h_\ga}{\partial x^i \partial x^j} dw^i\wedge dw^\barj.$$
We call $h_\ga$ the K\"ahler potential of $\omega_\ga$.
However we have a fixed moment map image $\mathcal P_\ga$ so that $h_\ga$ is 
determined only up to a constant. Since $\omega = \sum\limits_\ga \omega_\ga$ and
we have the equality of Minkowski sum $\mathcal P_{-K_S} = \sum\limits_\ga \mathcal P_\ga$
the K\"ahler potential of $\omega$ is equal to 
$\sum\limits_\ga \phi_\ga$ up to a constant. This implies on $Orb_{C(S)}(H^c, p)$
\begin{equation}\label{pote1}
\omega_0^m = e^{-\sum\limits_\ga h_\ga}  
\big(\sqrt{-1}\,dw^1 \wedge dw^{\overline 1}\big) \wedge \cdots \wedge \big(\sqrt{-1}\, dw^m \wedge dw^{\overline m}\big).
\end{equation}
If we set
$$ f_\ga = h_\ga + \phi_\ga$$
then 
\begin{equation}\label{pote2}
 \omega_\ga + \sqrt{-1}\partial\barpartial \phi_\ga
= \sqrt{-1} \frac{\partial^2 f_\ga}{\partial x^i \partial x^j} dw^i\wedge dw^\barj.
\end{equation}
By \eqref{pote1} and \eqref{pote2} we obtain
\begin{equation}\label{pote3}
e^{-t\sum\limits_\ga\phi_\ga}\omega_0^m = e^{t\sum\limits_\ga f_\ga - (1-t)\sum\limits_\ga h_\ga} 
\big(\sqrt{-1}\,dw^1 \wedge dw^{\overline 1}\big) \wedge \cdots \wedge \big(\sqrt{-1}\, dw^m \wedge dw^{\overline m}\big)
\end{equation}
and
\begin{equation}\label{pote4}
(\omega_\ga + \sqrt{-1}\partial\barpartial \phi_\ga)^m 
= \det\left(\frac{\partial^2 f_\ga}{\partial x^i \partial x^j}\right)
\big(\sqrt{-1}\,dw^1 \wedge dw^{\overline 1}\big) \wedge \cdots \wedge \big(\sqrt{-1}\, dw^m \wedge dw^{\overline m}\big).
\end{equation}
Since both $W_\ga(f_\ga)$ and $g_\ga + W_\ga(\phi_\ga)$ are Hamiltonian functions of
$W_\ga$ with respect to $\omega_\ga + \sqrt{-1}\partial\barpartial\phi_\ga$ we have
$$ W_\ga(f_\ga) + C_\ga = g_\ga + W_\ga(\phi_\ga)$$
for some constant $C_\ga$. Then normalizing $g_\ga$ so that
$$ \int_{Orb_S(H^c, \barp)} e^{g_\ga}\omega_\ga^m = 1$$
we obtain
$$ \int_{Orb_S(H^c, \barp)} e^{g_\ga + W_\ga(\varphi_\ga)}(\omega_\ga + \sqrt{-1}\partial\barpartial\phi_\ga)^m
= 1$$
and 
\begin{eqnarray*}
 \int_{Orb_S(H^c, \barp)} e^{W_\ga(f_\ga) + C_\ga} (\omega_\ga + \sqrt{-1}\partial\barpartial\phi_\ga)^m
&=& \int_{\bfR^m} e^{W_\ga(f_\ga) + C_\ga} \det\left(\frac{\partial^2 f_\ga}{\partial x^i \partial x^j}\right)dx\\
&=& e^{C_\ga}\int_{\mathcal P_\ga} e^{\langle W_\ga, p\rangle} dp.
\end{eqnarray*}
Thus we obtain $C_\ga = -\log Vol_{W_\ga}(\mathcal P_\ga)$, and therefore
\begin{equation}\label{pote5}
\frac{e^{W_\ga(f_\ga)}}{Vol_{W_\ga}(\mathcal P_\ga)} = e^{g_\ga + W_\ga(\phi_\ga)}
\end{equation}
where
$$ Vol_{W_\ga}(\mathcal P_\ga) = \int_{\mathcal P_\ga} e^{\langle W_\ga, p\rangle} dp.$$
From \eqref{pote3}, \eqref{pote4} and \eqref{pote5}, the Monge-Amp\`ere equation \eqref{MA} reduces to the real Monge-Amp\`ere equation
$$ \frac{e^{W_1(f_1)}}{Vol_{W_1}(\mathcal P_1)} \det\left(\frac{\partial^2 f_1}{\partial x^i \partial x^j}\right)
= \cdots = 
 \frac{e^{W_k(f_k)} }{Vol_{W_k}(\mathcal P_k)} \det\left(\frac{\partial^2 f_k}{\partial x^i \partial x^j}\right)
 = e^{t\sum\limits_\ga f_\ga - (1-t)\sum\limits_\ga h_\ga}.$$
 For the rest of the proof  the same arguments as in \cite{Hultgren17} applies.
 This completes the proof of Theorem \ref{toricsoliton}.
 \end{proof}
In \cite{TianZhu02} and \cite{Wang-Zhu04} it is shown that for a toric Fano manifold $\Fut_W$ is the derivative of
a proper convex function at $W \in \mathfrak k$ and that there exists a unique soliton vector field $W$, i.e. $\Fut_W = 0$.
The same arguments apply in our coupled Sasaki Ricci-soliton case to find $W  \in \mathfrak k$ 
with $\Fut_{W,\cdots,W} = 0$. Thus we obtain the following corollary to Theorem \ref{toricsoliton}.
\begin{corollary}\label{toricsoliton2}
Let $S$ be a compact toric Sasaki manifold with the basic first Chern class $c_1^B(S)$
positive, and $c_1^B(S)=(\sum\limits_{\ga=1}^k\gamma_\ga)/2\pi$ be a basic decomposition.
Then there exists a Killing potential $W$ such that for $W_1 = \cdots = W_k =W$ we have coupled
Sasaki Ricci-solitons.
\end{corollary}
General uniqueness result modulo automorphisms was established for coupled K\"ahler-Einstein metrics in
\cite{HultgrenWittNystrom18}. The case for Sasaki-Ricci solitons will necessitate the 
pluripotential theory for Sasaki manifolds, and is beyond the scope of this paper.
\section{Appendix}
Let $L \to X$ be an ample line bundle over a compact complex manifold $X$. We choose a Hermitian metric $h$ of $L$ such that its curvature form $\omega$ is a positive form. Suppose that
we have a Hamiltonian action of a torus $T$ on $X$. The moment map for the torus action is defined up to a translation. This ambiguity depends on the choice of lifting of the action on $X$ to $L$. 
However for the anti-canonical line bundle $L = K_X^{-1}$ we have the standard lifting, namely
the action induced by the push-forward. We call the moment map of $K_X^{-1}$ for Fano manifold
corresponding to the push-forward 
the {\it standard moment map}. Similarly we have the standard moment map for $K_S^{-1}$ 
for a compact Sasaki manifold $S$ with $c_1^B(S) > 0$.
\begin{theorem} Let $S$ (resp. $X$) be a toric Sasaki manifold with $c_1^B(S) > 0$ (resp. toric Fano manifold).
The moment map given by the Hamiltonian functions $u$ in \eqref{moment1} is the standard one
for $K_S^{-1}$ (resp. $K_X^{-1}$). 
\end{theorem}
\begin{proof}
We first consider the case of toric Fano manifold $X$. Choose a K\"ahler form 
$\omega = \sqrt{-1} g_{i\barj} dz^i \wedge dz^\barj$
with positive Ricci form $\rho:= \Ric(\omega)$. Then $\det g$ gives a Hermitian metric on $K_X^{-1}$
and the connection form $\theta$ on the associated principal $\bfC^\ast$-bundle is given by
$\theta = \zeta^{-1} d\zeta + (\det g)^{-1} \partial \det g$ where $\zeta$ is the fiber coordinate of the
$\bfC^\ast$-bundle, and its curvature is $\rho = \sqrt{-1}\barpartial \theta$. The Hamiltonian function 
in terms of $\rho$ for an element 
$V \in \mathfrak t$ is given by $- \theta( V) = -\mathrm{div}_g V$, i.e.
minus the divergence of $V$ with respect to $g$. Thus the barycenter of the moment polytope with respect
 to the K\"ahler form $\rho$ is given by
$-\int_X \mathrm{div}_g V \rho_g^m$. By Proposition 2.3 in \cite{futakimorita85}, this is equal to
$-\int_X V(F)\ \omega^m$. The Hamiltonian function $u_V$ in terms of $\omega$ satisfies
$\Delta_F u_V = - u_V$. Thus $-\int_X V(F)\ \omega^m = \int_X u_V\ \omega^m$. 
This shows the standard moment map have the same barycenter as the one given by 
the Hamiltonian functions in \eqref{moment1}. The case of Sasaki manifolds is similar.
This completes the proof.
\end{proof}



\end{document}